\documentclass[11pt,a4paper]{article}

\usepackage{amssymb, amsmath, amsthm}

\usepackage[cm]{fullpage}
\usepackage[english]{babel}
\usepackage[T1]{fontenc}
\usepackage[pdftex]{graphicx}
\usepackage{tikz}
\usepackage{xcolor}
\usepackage{hyperref}
\usepackage{booktabs}
\usepackage{subcaption}

\newtheorem{thm}{Theorem}

\newtheorem{prop}{Proposition}

\newtheorem{ex}{Example}
\newcommand{\matder}[2]{
\ifx&#1&\empty
\left(\dfrac{\partial}{\partial t}#2\dfrac{\partial}{\partial x}\right)
\else
\left(\dfrac{\partial}{\partial t}#2\dfrac{\partial}{\partial x}\right)^#1
\fi
}

\newcommand{\Prob}{\mathrm{{P}}}

\title{Fractional material derivative: pointwise representation and a finite volume numerical scheme}
\author{\L ukasz P\l ociniczak\thanks{Faculty of Pure and Applied Mathematics, Hugo Steinhaus Center, Wroc{\l}aw University of Science and Technology, Wyb. Wyspia{\'n}skiego 27, 50-370 Wroc{\l}aw, Poland}$\;^,$\footnote{Email: lukasz.plociniczak@pwr.edu.pl}\and Marek A. Teuerle$^*$}
\date{}

\begin{document}
\maketitle

\begin{abstract}
The fractional material derivative appears as the fractional operator that governs the dynamics of the scaling limits of L\'evy walks - a stochastic process that originates from the famous continuous-time random walks. It is usually defined as the Fourier-Laplace multiplier, therefore, it can be thought of as a pseudo-differential operator. In this paper, we show that there exists a local representation in time and space, pointwise, of the fractional material derivative. This allows us to define it on a space of locally integrable functions which is larger than the original one in which Fourier and Laplace transform exist as functions.

We consider several typical differential equations involving the fractional material derivative and provide conditions for their solutions to exist. In some cases, the analytical solution can be found. For the general initial value problem, we devise a finite volume method and prove its stability, convergence, and conservation of probability. Numerical illustrations verify our analytical findings. Moreover, our numerical experiments show superiority in the computation time of the proposed numerical scheme over a Monte Carlo method applied to the problem of probability density function's derivation.\\


\end{abstract}

\section{Introduction}
In various range of situations spanning from stochastic processes to partial differential equations both in theoretical and applied view, one frequently encounters nonlocal operators. Of the large family of them, fractional derivatives of a number of types are probably the most important examples. The concept of the classical integer-order derivative can be generalized in various directions, suggested by a growing collection of new experimental evidence of its usefulness and purely mathematical interest. Of course, these generalizations have to be constructed in a specific way to preserve the well-definiteness and be natural in some sense. There is an important result stating that, in order to be meaningful, the definition of the fractional derivative must involve a singular kernel \cite{diethelm2020fractional}. In this way, the well-known Riemann-Liouville, Caputo, Riesz, and other variants of the fractional derivative can be devised \cite{teodoro2019review}. 

The fractional material derivative was introduced in \cite{sokolov2003towards} to provide a deterministic approach in the description of L\'evy Walks (LW), which are stochastic processes that provide a model for superdiffusion. In physics, anomalous diffusion is a phenomenon in which the mean squared displacement (MSD) of a particle deviates in time from the classical linear dependence \cite{MetzlerKlafter}. This dynamics can be slower (subdiffusion) or faster (superdiffusion). The importance of anomalous diffusion comes from its occurrence in many different settings in turbulent flow \cite{Hen02}, plasma physics \cite{Del05}, material science \cite{Mul96}, viscoelasticity \cite{Amb96}, biology \cite{Sun17}, hydrology \cite{El04,El20}, and biophysics \cite{Kou08}. The mathematical description of anomalous diffusion is usually devised in a dual way: deterministic and stochastic. One can investigate a corresponding fractional Fokker-Planck partial differential equation for the probability density function (PDF) \cite{metzler1999deriving}, or construct a random motion that describes the moving particle \cite{bouchaud1990anomalous}. For the latter, the Continuous-Time Random Walk (CTRW) framework provides a versatile device for obtaining such processes, and we revise its basics in the next section to provide a detailed link to LWs and fractional material derivative. On the other hand, in the deterministic approach, the anomalous character of the phenomenon manifests itself as a nonlocal operator appearing in the corresponding fractional Fokker-Planck equation. For example, the subdiffusive character of the process can be a result of the phenomenon of long memory or heavy-tailed waiting time, which can be modeled by the Caputo fractional time derivative \cite{plociniczak2015analytical}. In the superdiffusive regime LW provides an interesting model in which spatial and temporal variables are strongly dependent. In the deterministic counterpart, it leads to the fractional material derivative. This operator is defined as a multiplier in the Fourier-Laplace sense, which indicates its manifestly nonlocal character. 

The main purpose of this paper is to provide a pointwise representation of the fractional material derivative, which helps us to deduce several important properties that may be difficult to observe otherwise. In particular, our main result states that the fractional material derivative is itself a Riemann-Liouville operator acting in the direction of the characteristic $x\pm t = C$. This substantially enlarges the natural space on which the fractional material derivative can be defined (we require only local integrability). From there, we are able to state explicit solutions to some model partial differential equations involving the material derivative, and provide a sufficient and necessary condition for solvability of the PDFs. In addition, we use our pointwise representation to construct a conservative finite-volume numerical method that can be used for a variety of PDE problems. We prove the stability and convergence of the method and illustrate our results with several numerical examples. 

Concerning numerical methods for the fractional derivative (or differential equations with this operator), the literature is very rich, and we will mention only some papers relevant to the present topic. A broad overview of numerical methods for fractional operators can be found in \cite{Li19,Li19a}. Specifically, in what follows, we use the L1 method to discretize the Riemann-Liouville fractional derivative. Essentially, it is a piecewise linear approximation method that is very versatile and monotone, leading to convenient analysis and implementation \cite{Li19}. This method has been applied to a variety of differential equation problems and has been thoroughly analyzed (see, for example, \cite{jin2016analysis,Plo21,plociniczak2023linear,Lia18}). Another popular method of discretizing the Riemann-Liouville derivative is based on the observation that it is itself a derivative of a convolution. The so-called convolution quadratures (CQ) inherit many useful properties from the continuous case and can be evaluated cheaply and quickly \cite{schadle2006fast,lubich2004convolution}.

In the next section, we revise some basic ideas and definitions concerning LWs to provide a motivation for considering the fractional material derivative. Section 3 contains our main result on point-wise representation and various corollaries. In Section 4 we construct the finite-volume numerical method and prove its stability and convergence. We conclude the paper with a collection of concrete numerical illustrations that verify our previous findings and indicate some possible future research directions. 

\section{A link between fractional material derivatives and stochastic processes}

The dynamics driven by a fractional material derivative can be derived from the well-known stochastic model of CTRW that was introduced by Weiss and Montroll in 1965 \cite{Montroll}. It generalizes the classical random walk \cite{Pearson, Feller} by allowing both the waiting times for jumps and jumps themselves to be random. Usually, it is assumed that every waiting time and the respective jump are independent (uncoupled). Moreover, all pairs of waiting times and jumps are mutually independent. Due to its importance in other parts, it is necessary to note that the CTRW trajectories are \textit{c\'adl\'ag}, i.e. they are right-continuous and have left limits everywhere in the domain, therefore, the trajectories are discontinuities due to the jumps. Formally, let us define a CTRW process as ${X}(t)=\sum_{i=1}^{N(t)}{J}_i$, where $\{(T_i,{J}_i)\}_{i\geq 1}$ is a sequence of independent and identically distributed (IID) random vectors of waiting times and jumps, both independent of each other, while $N(t)=\max\{n:T_1+T_2+\ldots+T_n\leq t\}$ is the renewal process that counts the number of jumps until time $t$. The sample trajectory of a walker that follows CTRW starts at 0 and then consists of waiting time $T_1$ that is followed by a respective jump $J_1$, then the next waiting time $T_2$ is followed by the next jump $J_2$, and then the scheme repeats itself. 

Now, CTRWs are well-established mathematical models \cite{KlafterSokolov,Komorowski2005,MetzlerKlafter} with a number of physical applications: from their initial use in modeling charge carrier transport in amorphous semiconductors \cite{Scher}, through subsurface tracer dispersion \cite{Scher2}, electron transfer \cite{Nelson}, behavior of dynamical systems \cite{Bologna},  noise in plasma devices \cite{Chechkin2002}, self-diffusion in micelles systems \cite{Ott}, models in gene regulation \cite{Lomholt2005} to dispersion in turbulent systems \cite{Solomon1993}, to name only a few. A proper choice of waiting times and jumps in the CTRW scheme can lead to various anomalous diffusion characteristics. For example, a CTRW model with infinite-variance waiting times and finite-variance jumps is the stochastic argument in the derivation of a fractional Fokker-Planck equation (FFPE) \cite{MetzlerKlafter} that is one of the models for subdiffusion. On the other hand, a superdiffusive CTRW scheme can be obtained by taking infinite-variance jumps and finite-variance waiting times. 

To introduce a family of LW processes \cite{sokolov2003towards,Klafter2, Klafter1, Magdziarz2012, Teuerle2012} one should consider a class of \textit{coupled} CTRWs that assume heavy-tailed jumps and a strong dependence between waiting times and respective jumps. Namely, for the heavy-tailed waiting times $\lim_{t\to \infty}t^\alpha\Prob(T_i>t)=c$  with the index $\alpha\in(0;1)$ and some positive constant c such that $\Prob(T_i>0)=1$, we construct respective jumps as 
 \begin{equation}
 {J}_i={V}_i\, T_i,
 \label{def:jumps}
 \end{equation}
where $\{V_i\}_{i\geq 1}$ is an IID sequence of random directions satisfying $\Prob(V_i=+1)=p=1-\Prob(V_i=-1)$ with $p\in[0;1]$. Random variables $V_i$'s control the direction (left/right) of every jump. What is important is that the jump lengths, due to the relation in \eqref{def:jumps}, are equal to the respective waiting time lengths, and this is true for every pair $(T_i,{J}_i)$.
We define the \textit{wait-first LW}, \textit{jump-first LW} and \textit{standard (continuous) LW}, respectively, as follows:
 \begin{equation}
	{L}_{WF}(t)=\sum_{i=1}^{N(t)}{J}_i,\qquad
	\label{def:wfLW}
\end{equation}
\begin{equation}
	{L}_{JF}(t)=\sum_{i=1}^{N(t)+1}{J}_i,\qquad	
	\label{def:jfLW}
\end{equation}
\begin{equation}
	{L}(t)=\sum_{i=1}^{N(t)}{J}_i+\left( t-\sum_{i=1}^{N(t)}T_i\right){V}_{N(t)+1}.
	\label{def:LW}
\end{equation}

\begin{figure}[!h]
\begin{center}
\includegraphics[scale=0.75]{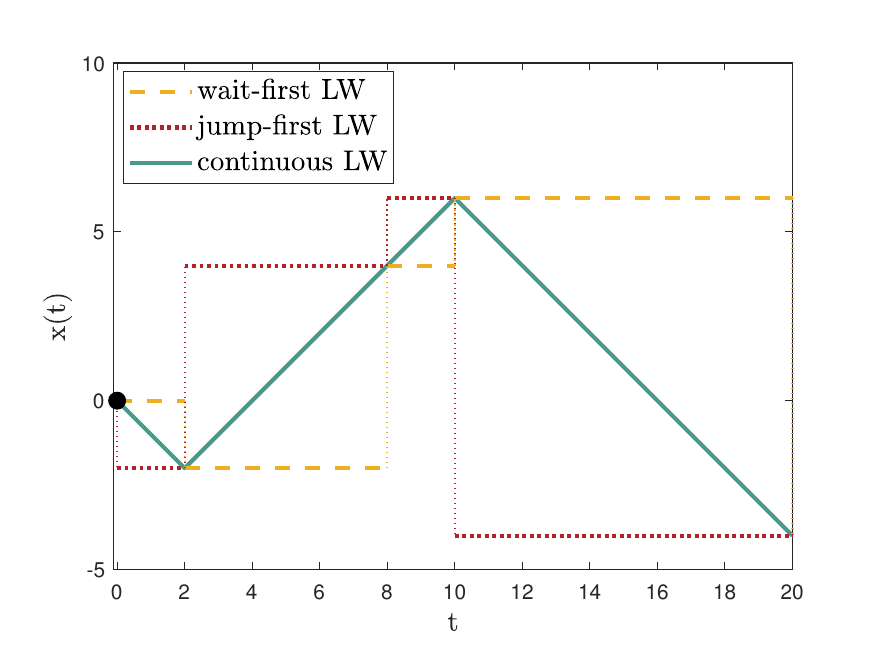}
\caption{The sample one-dimensional trajectories of 3 types of LWs: wait-first, jump-first and continuous LW. Here, we assume that the probability of choosing the new direction being left or right at the renewals is equal to 0.5. Later on, we focus only on cases, where this probability of going left (or right) is equal to 1.}
\label{fig:trajectory}
\end{center}
\end{figure}

The sample trajectories of the LWs considered are presented in Fig. \ref{fig:trajectory}. The trajectories of wait-first LW \eqref{def:wfLW} follow a typical CTRW pattern, while for jump-first LW \eqref{def:jfLW} every waiting time is preceded by a coupled jump. 

The standard LW \eqref{def:LW}, originally proposed by \cite{Klafter1}, has continuous trajectories and is constructed as a linear interpolation between wait-first LW and jump-first LW. Obviously, this leads to renewals at the same epochs of time as renewals of the underlying processes. Furthermore, the distances between renewal times are heavy-tailed with the same index $\alpha$ as the underlying wait-first and jump-first LWs.

To establish the direct link between LWs and the fractional material derivative, we need to analyze the functional convergence in the Skorokhod space \cite{Whitt} of LW schemes \eqref{def:wfLW}-\eqref{def:LW}. The following results hold \cite{MagdziarzZorawik2017, Magdziarz2015LWFCLT,Jurl_FCLT,MagdziarzZorawik,Teuerle2015}:
\begin{equation}
\frac{{L}_{WF}(nt)}{n}\xrightarrow{n\rightarrow\infty} \left(L^-_{\alpha}(S_\alpha^{-1}(t))\right)^+,
\label{def:fcltWFLW}
\end{equation}
\begin{equation}
\frac{{L}_{JF}(nt)}{n}\xrightarrow{n\rightarrow\infty} L_{\alpha}(S_\alpha^{-1}(t)),
\label{def:fcltJFLW}
\end{equation}
\begin{equation}
\frac{{L}(nt)}{n}\xrightarrow{n\rightarrow\infty} \left\{
\begin{tabular}{lcr}
$\left(L^-_{\alpha}(S_\alpha^{-1}(t))\right)^+$ & if & $t\in \mathcal{R}_\alpha$,\\
$\left(L^-_{\alpha}(S_\alpha^{-1}(t))\right)^+ +\frac{t-G(t)}{H(t)-G(t)}\left(L_{\alpha}(S_\alpha^{-1}(t))-\left(L^-_{\alpha}(S_\alpha^{-1}(t))\right)^+\right)$ & if & $t\notin \mathcal{R}_\alpha$,\\
\end{tabular}
\right.
\label{def:fcltLW},
\end{equation}
where $\mathcal{R}_\alpha(S_\alpha(t))=\{S_\alpha(t):t\geq 0\}\in[0,\infty)$.
Here, $L_\alpha(s)$ and $S_\alpha(t)$ are c\'adl\'ag scaling limit of the partial sums of jumps and waiting times, respectively. $L^-_\alpha(s)$ and $S^-_\alpha(t)$ denote left-continuous versions of $L_\alpha(s)$ and $S_\alpha(t)$, i.e. $L^-_\alpha(s)=\lim_{s\to t^-} L_\alpha(s)$. The process $S_\alpha^{-1}(t)$ is the inverse subordinator to $S_\alpha(t)$ and is defined as $S_\alpha^{-1}(t)=\inf\{\tau:S_\alpha(\tau)>t\}$. Finally, $\left(L^-_{\alpha}(S_\alpha^{-1}(t))\right)^+$ denotes the right-continuous version of $L^-_{\alpha}(S_\alpha^{-1}(t))$ and $G(t)=S^-_\alpha(S_\alpha^{-1}(t+))$ is the moment of the previous jump of $\left(L^-_{\alpha}(S_\alpha^{-1}(t))\right)^+$ and $H(t)=S_\alpha(S_\alpha^{-1}(t+))$ is the moment of the next jump of $L_{\alpha}(S_\alpha^{-1}(t))$ and $S_\alpha(t+)=\lim_{s\to t^+}S_\alpha(s)$. 

Finally, it appears that all three PDFs $u(x,t)$ of the scaling limits of LWs \eqref{def:fcltWFLW}-\eqref{def:fcltLW} fulfill the following pseudo-differential equation:

\begin{equation}
\label{eqn:ProbabilisticProblem}
\left( p\matder{\alpha}{+}+(1-p)\matder{\alpha}{-}\right) u(x,t) = f(x,t), \quad u(x,0) = \delta(x), \quad x \geq 0,
\end{equation}
where the source function $f(x,t)$ depends on the type of LW \cite{Magdziarz2015LWFCLT,Teuerle2015} and
\begin{equation}
\mathcal{F}\mathcal{L}\left\{ \matder{\alpha}{\pm} u(x,t) \right\}=(s\mp i \xi)^\alpha \hat{U}(\xi,s)
\label{def:fracMatDer}
\end{equation}
denotes the Fourier-Laplace transform of the \textit{fractional material derivative} for a PDF $u(x,t)$.

In the literature \cite{MagdziarzZorawik2017, Magdziarz2015LWFCLT}, one can find an exact solution to the PDF derivation problem for the scaling limits of LWs, although in this work we want to establish a methodology for a more general problem and propose a stable and convergent scheme for numerical approximation that can also be accurate in the probabilistic cases mentioned above.

\section{Integro-differential form of the fractional material derivative}
The definition of the fractional material derivative using the Fourier-Laplace transform \eqref{def:fracMatDer} is very natural and straightforward. However, for numerical computation, it appears that it is more convenient to work in the original $x-t$ space rather than with transformed variables. The following result is interesting on its own. 

\begin{thm}[Pointwise representation]
Let $u$ be of $x$-Schwarz class and $t$-exponential on $\mathbb{R}\times\mathbb{R}_+$ so that the Fourier-Laplace transform is well-defined. Then, for $0<\alpha<1$ we have
\begin{equation}
    \label{eqn:MatDerInt}
    \matder{\alpha}{\pm} u(x,t) = \frac{1}{\Gamma(1-\alpha)}\matder{}{\pm} \int_0^t (t-s)^{-\alpha} u(x \mp (t-s),s) ds
\end{equation}
\end{thm}
\begin{proof}
We will consider only the "$+$" version of the fractional material derivative - the other case is analogous. Let $\widehat{U}(\xi,s)$ be the Fourier-Laplace transform of $u(x,t)$, that is
\begin{equation}
    \widehat{U}(\xi,s) = \int_0^\infty \left(\int_\infty^\infty u(x, t) e^{i \xi x} dx\right) e^{-s t} dt.
\end{equation}
Since $u$ is $x$-Schwartz and $t$-exponential, the above is well-defined and we can use inversion formulas. For each $s$ we have $\widehat{U}(\cdot,s)\in\mathcal{S}$ and $(s-i\xi)^\alpha$ is locally $\xi$ integrable, therefore,
\begin{equation}
    \matder{\alpha}{+} u(x,t) = \mathcal{L}^{-1} \left\{\frac{1}{2\pi}\int_{-\infty}^\infty (s-i\xi)^\alpha \widehat{U}(\xi,s) e^{-i\xi x} d\xi\right\}(t) =:  \mathcal{L}^{-1} \left\{J(s)\right\}(t).
\end{equation}
Now, we will compute the integral $J(s)$. First, change the variable $\eta = \xi + is$ to obtain
\begin{equation}
    J(s) = \frac{1}{2\pi}\int_{-\infty}^\infty (-i\eta)^\alpha \widehat{U}(\eta-is,s) e^{-i (\eta-is) x} d\eta = 
    \lim\limits_{\epsilon\rightarrow 0} \frac{1}{2\pi}\int_{-\infty}^\infty (-i\eta)^\alpha \widehat{U}(\eta-is,s) e^{-i (\eta-is) x} e^{-\epsilon |\eta|} d\eta,
\end{equation}
where we have introduced the convergence factor $e^{-\epsilon |\eta|}$. Now, we can unravel the definition of the Fourier transform and change the order of integration
\begin{equation}
    \begin{split}
        J(s) &= \lim\limits_{\epsilon\rightarrow 0} \frac{1}{2\pi}\int_{-\infty}^\infty \left(\int_{-\infty}^\infty (-i\eta)^{\alpha} e^{-\epsilon |\eta| - i(x-y)\eta} d\eta\right) e^{-(x-y)s} U(y,s)dy \\
        &= \lim\limits_{\epsilon\rightarrow 0} \frac{1}{2\pi}\int_{-\infty}^\infty \left(\frac{\partial}{\partial x}\int_{-\infty}^\infty (-i\eta)^{\alpha-1} e^{-\epsilon |\eta| - i(x-y)\eta} d\eta\right) e^{-(x-y)s} U(y,s)dy\\
        &=: \int_{-\infty}^\infty \lim\limits_{\epsilon\rightarrow 0} \left(\frac{\partial}{\partial x} K_\epsilon(x-y) \right) e^{-(x-y)s} U(y,s)dy, 
    \end{split}
\end{equation}
where we are left with evaluating the limit above. To this end, note that $K_\epsilon$ is the inverse Fourier transform of the $L^1(\mathbb{R})$ function.
\begin{equation}
    \widehat{K}_\epsilon(\eta) = (-i\eta)^{\alpha-1} e^{-\epsilon |\eta|}.
\end{equation}
It is known from harmonic analysis that, in the sense of tempered distributions, we have []
\begin{equation}
    \widehat{\left(\frac{1}{\Gamma(1-\alpha)} x^{-\alpha} H(x)\right)}(\eta) = (-i\eta)^{\alpha-1} = \lim\limits_{\epsilon\rightarrow 0} (-i\eta)^{\alpha-1} e^{-\epsilon |\eta|} = \lim\limits_{\epsilon\rightarrow 0} \widehat{K}_\epsilon(\eta),
\end{equation} 
whence
\begin{equation}
    \lim\limits_{\epsilon\rightarrow 0} K_\epsilon(x-y) = \frac{(x-y)^{-\alpha}}{\Gamma(1-\alpha)} H(x-y).
\end{equation}
This leads to
\begin{equation}
    \begin{split}
        J(s) 
        &= \frac{1}{\Gamma(1-\alpha)}\int_{-\infty}^\infty \frac{\partial}{\partial x} \left((x-y)^{-\alpha} H(x-y)\right) e^{-(x-y)s} U(y,s)dy \\
        &=\frac{1}{\Gamma(1-\alpha)} \left(\frac{\partial}{\partial x}\int_{-\infty}^x (x-y)^{-\alpha} e^{-(x-y)s} U(y,s)dy + \int_{-\infty}^x (x-y)^{-\alpha} e^{-(x-y)s}s U(y,s)\right),
    \end{split}
\end{equation}
where we have used the elementary formula for the derivative of a product. Now, we are ready to invert the Laplace transform. To this end, we note that from
\begin{equation}
    \mathcal{L} \left\{f'(t)\right\}(s) = s \mathcal{L} \left\{f(t)\right\}(s) - f(0) = s \mathcal{L} \left\{f(t)\right\}(s) - f(0) \mathcal{L} \left\{\delta(t)\right\}(s),
\end{equation}
it follows that
\begin{equation}
    \mathcal{L}^{-1} \left\{e^{-cs}s \mathcal{L} \left\{f(t)\right\}(s)\right\}(t) = \left(f'(t-c) + f(0) \delta(t-c)\right) H(t-c).
\end{equation}
Therefore,
\begin{equation}
    \begin{split}
        \matder{\alpha}{+} u(x,t) &= \frac{1}{\Gamma(1-\alpha)} \left(\frac{\partial}{\partial x}\int_{x-t}^x (x-y)^{-\alpha} u(y,t-(x-y))dy \right. \\
        &\left.+ \int_{x-t}^x (x-y)^{-\alpha} \left(u_t(y, t-(x-y)) + u(y,0) \delta(t-(x-y))\right)dy\right).
    \end{split}
\end{equation}
Now, evaluating the integral with the Dirac delta yields the term $t^{-\alpha} u(x-t,0)$ while using the Leibniz rule for the derivative of the integral returns
\begin{equation}
    \int_{x-t}^x (x-y)^{-\alpha} u_t(y, t-(x-y)) dy = \frac{\partial}{\partial t} \int_{x-t}^x (x-y)^{-\alpha} u(y, t-(x-y)) dy - t^{-\alpha} u(x-t,0),
\end{equation}
which produces the same with opposite sign. Hence,
\begin{equation}
    \matder{\alpha}{+} u(x,t) = \frac{1}{\Gamma(1-\alpha)} \matder{}{+}\int_{x-t}^x (x-y)^{-\alpha} u(y,t-(x-y))dy.
\end{equation}
The substitution $s = y + t-x$ yields the conclusion. 
\end{proof}
Notice that the formula (\ref{eqn:MatDerInt}) is well-posed assuming merely that $u$ is locally integrable without any requirements on its decay at infinity. Henceforth, we can \emph{redefine} the fractional material derivative by (\ref{eqn:MatDerInt}) on such a space enlarging its domain. Further, as can be inferred from the integral in (\ref{eqn:MatDerInt}), the calculation of the material derivative at $(x,t)$ requires knowledge of all past values of $u$ along the "characteristic" $x \mp t = C_\pm$ for some $C\in\mathbb{R}$ (see Fig. \ref{fig:Domain}). As we shall see in the following, this implies that the fractional material derivative is essentially the Riemnann-Liouville derivative computed in the characteristic direction.  

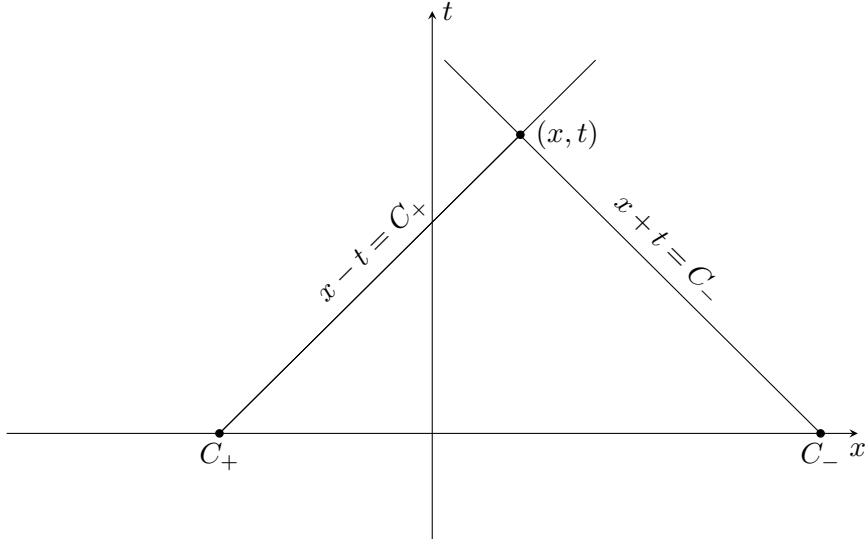
\begin{figure}[!h]
\centering
\begin{tikzpicture}[scale = 1.4]
    \draw[->, >= stealth] (-4,0) -- (4,0) node[right, below] {$x$};
    \draw[->, >= stealth] (0, -1) -- (0,4) node[right] {$t$}; 
    
    \draw[fill] (-2, 0) -- +(45:5) node[midway, sloped, above, xshift = -10pt] {$x-t=C_+$};
    \draw[fill] (3.65, 0) -- +(135:5) node[midway, sloped, above, xshift = 10pt] {$x+t=C_-$};
    \draw[fill] (-2, 0) -- +(45:4) circle (1pt) node[right, xshift = +2pt] {$(x,t)$};
    \draw[fill] (-2, 0) circle (1pt) node[below] {$C_+$};
    \draw[fill] (3.65, 0) circle (1pt) node[below] {$C_-$};
\end{tikzpicture}
\caption{Lines of dependence of the fractional material derivative given by (\ref{eqn:MatDerInt}). }
\label{fig:Domain}
\end{figure}

We can immediately make some simple observations.
\begin{prop}
\label{prop:Properties}
Let $u, f \in L^1_{loc} (\mathbb{R}\times \mathbb{R}_+)$. Then, for $0<\alpha<1$ we have
\begin{enumerate}
    \item \emph{Continuity with respect to $\alpha$.}
    \begin{equation}
        \lim\limits_{\alpha \rightarrow 0^+} \matder{\alpha}{\pm} u(x,t) = u(x,t), \quad \lim\limits_{\alpha \rightarrow 1^-} \matder{\alpha}{\pm} u(x,t) = \matder{}{\pm}u(x,t).
    \end{equation} 
    \item \emph{Action on a travelling wave.} For $u(x,t) = U(x \mp t)$ we have
    \begin{equation}
        \matder{\alpha}{\pm} u(x,t) = \frac{t^{-\alpha}}{\Gamma(1-\alpha)} U'(x \mp t). 
    \end{equation}
    \item \emph{Solution of the PDE.} The unique solution of the problem
    \begin{equation}
        \label{eqn:MatDerPDE}
        \begin{cases}
            \matder{\alpha}{\pm} u(x,t) = f(x,t), & x\in\mathbb{R}, \quad t>0, \\
            \lim\limits_{t\rightarrow 0^+} I^{1-\alpha}_t u(x,t) = \varphi(x), &
        \end{cases}
    \end{equation}
    is given by
    \begin{equation}
        \label{eqn:MatDerPDESol}
        u(x,t) = \frac{t^{\alpha-1}}{\Gamma(\alpha)} \varphi(x \mp t) + \frac{1}{\Gamma(\alpha)}\int_0^t (t-s)^{\alpha-1} f(x \mp (t - s), s) ds. 
    \end{equation}
    \item \emph{Conservation Law.} For any $a$, $b \in\mathbb{R}$ we have 
    \begin{equation}
        \begin{split}
            \frac{1}{\Gamma(1-\alpha)} &\frac{d}{dt} \int_0^t (t-s)^{-\alpha} \left(\int_{a-(t-s)}^{b-(t-s)} u(y,s) dy\right) ds \\
            &+ \frac{1}{\Gamma(1-\alpha)} \int_0^t (t-s)^{-\alpha} \left(u(b-(t-s),s) - u(a-(t-s),s) \right)ds = \int_a^b f(x,t). 
        \end{split}
    \end{equation}
\end{enumerate}
\end{prop}  
\begin{proof}
The first statement follows easily from the properties of the Riemann-Liouville fractional derivative (see, for example, \cite{Li19}). For the second, it is sufficient to observe that $u(x \mp (t-s),s) = U(x \mp (t-s) \mp s) = U(x \mp t)$. This statement follows after calculating the value of the remaining integral. 

For the third part of the proposition, we will use the classical method of characteristics. As above, we will focus only on the "$+$" case, the other is solved analogously. Fix a constant $C\in\mathbb{R}$ and consider the characteristic $x - t = C$. Define the value of the fractional integral on the characteristic
\begin{equation}
    \widetilde{U}(x,t) := \frac{1}{\Gamma(1-\alpha)} \int_0^t (t-s)^{-\alpha} u(x -
    (t-s),s) ds, \quad U(t) := \widetilde{U}(t+C, t).
\end{equation}
Then, we have
\begin{equation}
    \frac{d}{dt}U(t) = \matder{}{+}\widetilde{U}(x,t), \quad x = t + C.
\end{equation}
Therefore, according to our representation (\ref{eqn:MatDerInt}), the PDE (\ref{eqn:MatDerPDE}) transforms into an ordinary fractional equation $D^\alpha U(t) = f(t+C,t)$. The existence and uniqueness of its solution quickly follows from general theory. We can also apply the fractional integral operator on both sides to obtain
\begin{equation}
    U(t) = \frac{t^{\alpha-1}}{\Gamma(1-\alpha)} \left[I^{1-\alpha} U(t)\right]_{t=0} + \frac{1}{\Gamma(\alpha)} \int_0^t (t-s)^{\alpha-1} f(s+C,s)ds. 
\end{equation}
Using the initial condition (\ref{eqn:MatDerPDE}) and recalling that $x - t = C$ we conclude the proof. 
\end{proof}
Having the exact form of the solution to (\ref{eqn:MatDerPDESol}) allows us to easily compute some PDFs of various CTRWs. As appears, the requirement of the solution to be the probability density forces some stringent conditions on the source function $f$. 

\begin{thm}[One-sided probability density problem]
\label{thm:Density}
Let $f(\cdot, t)\in L^1(\mathbb{R})$ for each $t>0$ and $f(x,\cdot) \in L^1_{loc}(\mathbb{R}_+)$ for each $x\in\mathbb{R}$. Furthermore, assume that $\varphi$ defined in (\ref{eqn:MatDerPDESol}) is in $L^1(\mathbb{R})$. Then, the problem
\begin{equation}
    \label{eqn:DensityProb}
    \begin{cases}
        \matder{\alpha}{\pm} u(x,t) = f(x,t), & x\in\mathbb{R}, \quad 0<t\leq T, \quad 0<\alpha<1, \\
        \displaystyle{\int_{-\infty}^\infty} u(x,t) dx = 1, \quad u(x,t) \geq 0, & \\
        \lim\limits_{x\rightarrow \pm \infty} u(x,t) = 0, &\\
        \lim\limits_{t\rightarrow 0^+} u(x,t) = \delta(x), & \\
    \end{cases}
\end{equation}
has a unique solution given by
\begin{equation}
    \label{eqn:DensitySol}
    u(x,t) = \frac{1}{\Gamma(\alpha)} \int_0^t (t-s)^{\alpha-1} f(x \mp (t-s),s) ds,
\end{equation}
if and only if $f(x,t)\geq 0$ and 
\begin{equation}
    \label{eqn:fCond}
    \int_{-\infty}^\infty f(x,t) dx = \frac{t^{-\alpha}}{\Gamma(1-\alpha)}.	
\end{equation}
The limit in the initial condition are interpreted in the distributional sense. 
\end{thm}
\begin{proof}
As above, we consider only the "$+$" case. Assume that $u$ is the solution (\ref{eqn:DensityProb}). Integrating the PDE with respect to $x$ over $\mathbb{R}$ with the integro-differential representation (\ref{eqn:MatDerInt}) we obtain
\begin{equation}
    \begin{split}
        \frac{1}{\Gamma(1-\alpha)} &\frac{d}{dt} \int_0^t (t-s)^{-\alpha} \left(\int_{-\infty}^\infty u(x - (t-s),s) dx \right) ds \\
        &+ \frac{1}{\Gamma(1-\alpha)} \int_0^t (t-s)^{-\alpha} \int_{-\infty}^\infty \frac{\partial}{\partial x} u(x - (t-s), s) ds = \int_{-\infty}^\infty f(x, t)dx. 
    \end{split}
\end{equation}
Since for any time $t\geq 0$ the $x-$integral over the whole space is equal to $1$ and $u(\pm\infty, t) = 0$ we further have
\begin{equation}
    \frac{1}{\Gamma(1-\alpha)} \frac{d}{dt} \int_0^t (t-s)^{-\alpha} ds = \int_{-\infty}^\infty f(x, t)dx.
\end{equation}
The left-hand side of the above can be evaluated to obtain (\ref{eqn:fCond}) proving that this is the necessary condition. 

Now, assume that (\ref{eqn:fCond}) holds. We know from Proposition \ref{prop:Properties} that the solution of the initial value problem is given by (\ref{eqn:MatDerPDESol}). Integrating it over the whole $x-$space and using our assumption on $f$ we obtain
\begin{equation}
    \int_{-\infty}^\infty u(x,t) dx = \frac{t^{\alpha-1}}{\Gamma(\alpha)} \int_{-\infty}^\infty \varphi(x - t) dx + \frac{1}{\Gamma(\alpha)\Gamma(1-\alpha)} \int_0^t (t-s)^{\alpha-1} s^{-\alpha}ds. 
\end{equation}
The integral on the right can be computed exactly in terms of the beta function
\begin{equation}
    \frac{1}{\Gamma(\alpha)\Gamma(1-\alpha)} \int_0^t (t-s)^{\alpha-1} s^{-\alpha}ds = \frac{1}{\Gamma(\alpha)\Gamma(1-\alpha)} \int_0^1 (1-s)^{\alpha-1} s^{-\alpha}ds = \frac{B(\alpha, 1-\alpha)}{\Gamma(\alpha)\Gamma(1-\alpha)} = 1. 
\end{equation}
Whence,
\begin{equation}
    \int_{-\infty}^\infty u(x,t) dx = \frac{t^{\alpha-1}}{\Gamma(\alpha)} \int_{-\infty}^\infty \varphi(x) dx + 1.
\end{equation}
Now, since $\varphi$ is integrable it has to vanish because otherwise, the right-hand side would become unbounded as $t\rightarrow 0^+$. Hence, $\int u(x,t)dx = 1$ and we also have (\ref{eqn:DensitySol}). We are left in showing that the initial condition is the Dirac delta. To this end, take any $\psi\in C_c^\infty(\mathbb{R})$, multiply it by (\ref{eqn:MatDerPDESol}), and integrate to arrive at
\begin{equation} 
    \int_{-\infty}^\infty u(x,t) \psi(x) dx = \frac{1}{\Gamma(\alpha)} \int_0^t (t-s)^{\alpha-1} \left( \int_{-\infty}^\infty f(x - (t-s), s) \psi(x) dx \right) ds
\end{equation}	
and thus (\ref{eqn:fCond}) is sufficient. However, by condition (\ref{eqn:fCond}) we can write
\begin{equation}
    \psi(0) \int_{-\infty}^\infty f(x - (t-s), s) \psi(x) dx = \psi(0) \frac{t^{-\alpha}}{\Gamma(1-\alpha)},
\end{equation}
which implies
\begin{equation}
    \int_{-\infty}^\infty f(x - (t-s), s) \psi(x) dx = \int_{-\infty}^\infty f(x - (t-s), s) \left(\psi(x)-\psi(0)\right) dx + \psi(0)\frac{s^{-\alpha}}{\Gamma(1-\alpha)}.
\end{equation}
Therefore, once again utilizing the definition of beta function, we obtain
\begin{equation}
    \int_{-\infty}^\infty u(x,t) \psi(x) dx = \psi(0) + \frac{1}{\Gamma(\alpha)} \int_0^t (t-s)^{\alpha-1} \left( \int_{-\infty}^\infty f(x - (t-s), s) \left(\psi(x)-\psi(0)\right) dx\right) ds.
\end{equation}
Hence, we are left in showing that the double integral on the right-hand side vanishes for $t\rightarrow 0^+$. To see this, change the time variable $s = t z$ to have
\begin{equation}
    \begin{split}
        &\left|\frac{1}{\Gamma(\alpha)} \int_0^t (t-s)^{\alpha-1} \left( \int_{-\infty}^\infty f(x - (t-s), s) \left(\psi(x)-\psi(0)\right) dx\right) ds\right|\\
        &\leq \frac{t^{\alpha}}{\Gamma(\alpha)} \int_0^1 (1-z)^{\alpha-1} \left( \int_{-\infty}^\infty f(x - t(1-z), tz) \left|\psi(tx)-\psi(0)\right| dx\right) dz \\
        &\leq \sup_{x\in\mathbb{R}} |\psi(tx) -\psi(0)|\frac{t^{\alpha}}{\Gamma(\alpha)} \int_0^t (1-z)^{\alpha-1} \left( \int_{-\infty}^\infty f(x - t(1-z), tz) dx\right) dz.
    \end{split}
\end{equation}
Now, since the condition (\ref{eqn:fCond}) is satisfied we obtain 
\begin{equation}
    \begin{split}
        &\left|\frac{1}{\Gamma(\alpha)} \int_0^t (t-s)^{\alpha-1} \left( \int_{-\infty}^\infty f(x - (t-s), s) \left(\psi(x)-\psi(0)\right) dx\right) ds\right|\\
        &\leq \sup_{x\in\mathbb{R}} |\psi(tx) -\psi(0)|\frac{t^{\alpha}}{\Gamma(\alpha)} \int_0^t (1-z)^{\alpha-1} \left(\int_{-\infty}^\infty f(x, tz) dx\right) dz \\
        &= \sup_{x\in\mathbb{R}} |\psi(tx) -\psi(0)|\frac{1}{\Gamma(\alpha)\Gamma(1-\alpha)} \int_0^t (1-z)^{\alpha-1} z^{-\alpha} dz = \sup_{x\in\mathbb{R}} |\psi(tx) -\psi(0)|.
    \end{split}
\end{equation}
By the uniform convergence of $\psi$ the right-hand side of the above goes to $0$ as $t\rightarrow 0^+$. Therefore, the condition (\ref{eqn:fCond}) is sufficient. 
\end{proof}

\begin{ex}
\label{ex:waitFirstLW}
Consider the scaling limit of wait-first LW \cite{Bec04, Kotulski1995,Jurlewicz2012} given by \eqref{def:fcltWFLW}. Then, its PDF $u(x,t)$ satisfies (see \cite{Jurlewicz2012} , formula (5.10))
\begin{equation}
\label{eqn:ULWProblem}
\matder{\alpha}{+} u(x,t) = \frac{t^{-\alpha}}{\Gamma(1-\alpha)}\delta(x), \quad u(x,0) = \delta(x), \quad x \geq 0.
\end{equation}
A straightforward calculation verifies that the solvability condition \eqref{eqn:fCond} is satisfied, and hence by (\ref{eqn:DensitySol}) we have 
\begin{equation}
\label{eqn:solWaitFirstLW}
u(x,t) = \frac{1}{\Gamma(\alpha)\Gamma(1-\alpha)} \int_0^t (t-s)^{\alpha-1} \delta(x-(t-s)) s^{-\alpha} ds = \frac{\sin(\pi\alpha) }{\pi}x^{\alpha-1}(t-x)^{-\alpha} H(t-x), \quad x \geq 0,
\end{equation}
where $H$ is the Heaviside step function and we have used the reflection formula for gamma function $\Gamma(\alpha)\Gamma(1-\alpha) = \pi/\sin(\pi\alpha)$. This result has previously been obtained in the cited literature by different purely stochastic means. 
\end{ex}

\begin{ex}
\label{ex:jumpFirstLW}
For the jump-first version of the L\'evy Walk, the PDE of its scaling limit \eqref{def:fcltJFLW} is the following (see \cite{Jurlewicz2012}, formula (5.13))
\begin{equation}
\matder{\alpha}{+} u(x,t) = \frac{\alpha}{\Gamma(1-\alpha)}\int_{t}^\infty \delta(x-u) u^{-\alpha-1} du = \frac{\alpha x^{-\alpha-1}}{\Gamma(1-\alpha)} H(x-t), \quad u(x,0) = \delta(x),
\end{equation}
where $x \geq 0$. The solvability condition \eqref{eqn:fCond} is satisfied, since
\begin{equation}
\int_{-\infty}^\infty \frac{\alpha x^{-\alpha-1}}{\Gamma(1-\alpha)} H(x-t) dx = \frac{\alpha}{\Gamma(1-\alpha)} \int_t^\infty x^{-\alpha-1} dx = \frac{t^{-\alpha}}{\Gamma(1-\alpha)},
\end{equation}
hence by (\ref{eqn:DensitySol}) we have 
\begin{equation}
u(x,t) = \frac{\alpha}{\Gamma(\alpha)\Gamma(1-\alpha)} \int_0^t(t-s)^{\alpha-1} (x-(t-s))^{-\alpha-1} H(x-t) ds.
\end{equation}
and therefore
\begin{equation}
u(x,t) = \frac{\alpha H(x-t)}{\Gamma(\alpha)\Gamma(1-\alpha)} \int_0^t z^{\alpha-1} (x-z)^{-\alpha-1} dz = \frac{\alpha t^\alpha H(x-t)}{\Gamma(\alpha)\Gamma(1-\alpha)} \int_0^1 w^{\alpha-1} (x-t w)^{-\alpha-1} dw, 
\end{equation} 
where we first substituted $z = t-s$, and later $z = tw$. The resulting integrand has a primitive, and hence
\begin{equation}
\label{eqn:solJumpFirstLW}
u(x,t) =  \frac{\alpha t^\alpha H(x-t)}{\Gamma(\alpha)\Gamma(1-\alpha)} \left[\frac{w^\alpha (x-t w)^{-\alpha}}{\alpha x}\right]_{w=0}^{w=1} =\frac{\sin(\pi\alpha)}{\pi} \frac{1}{x}\left(\frac{t}{x-t}\right)^\alpha H(x-t),
\end{equation}
which coincides with formula (5.12) in \cite{Jurlewicz2012} obtained by different methods.
\end{ex}

\begin{ex}
\label{ex:standardLW}
For the scaling limit of the standard L\'evy Walk \eqref{def:fcltLW} the PDE is the following (see \cite{Magdziarz2015LWFCLT}, Theorem 5.6 and Example 5.10)
\begin{equation}
\matder{\alpha}{+} u(x,t) = \frac{t^{-\alpha}}{\Gamma(1-\alpha)} \delta(x-t)  , \quad u(x,0) = \delta(x),
\label{standLWequation}
\end{equation}
where $x\geq 0$. The solvability condition \eqref{eqn:fCond} is satisfied, since
\begin{equation}
\int_{-\infty}^\infty \frac{ t^{-\alpha}}{\Gamma(1-\alpha)} \delta(x-t) dx = \frac{t^{-\alpha}}{\Gamma(1-\alpha)},
\end{equation}
hence by (\ref{eqn:DensitySol}) we have 
\begin{equation}
u(x,t) = \frac{\delta(x-t)}{\Gamma(\alpha)\Gamma(1-\alpha)} \int_0^t  (t-s)^{\alpha-1} s^{-\alpha} ds.
\end{equation}
Note that the last integral after two substitutions, $z = t-s$ and $z = tw$, can be rewritten as
\begin{equation}
u(x,t) = \delta(x-t)\frac{1}{\Gamma(\alpha)\Gamma(1-\alpha)} \int_0^1  w^{\alpha-1} (1-w)^{-\alpha} dw. 
\label{solvability3}
\end{equation} 
The on right hand side of (\ref{solvability3}) one can identify the integral over the probability distribution function of beta random variable $B(\alpha,1-\alpha)$.  Hence, we obtain the following solution
\begin{equation}
\label{eqn:solStandardLW}
u(x,t) =  \delta(x-t),\qquad x \geq 0,
\end{equation}
which is obvious from the point of view of the macroscopic stochastic interpretation and at the same time unapparent for the problem formulation in (\ref{standLWequation}).
\end{ex}

\section{Finite-volume upwind method}
The integro-differential form of the fractional material derivative (\ref{eqn:MatDerInt}) can be used to devise a numerical approximation scheme. Having in mind the problem of finding pdf (\ref{eqn:DensityProb}) we would like to construct a method that, at least approximately, is conservative and can deal with singular sources. The first natural choice is to use one of the finite-volume methods \cite{leveque2002finite} and here, we propose one of the simplest ones. To wit, for the uniform space-time mesh with step $h>0$
\begin{equation}
x_i = ih, \quad t_n = nh, \quad i\in\mathbb{Z}, \quad n\in\mathbb{N},
\end{equation}
we partition the real line into finite cells (volumes)
\begin{equation}
\mathbb{R} = \bigcup_{i = -\infty}^\infty (x_{i-\frac{1}{2}}, x_{i+\frac{1}{2}}], \quad x_{i \pm \frac{1}{2}} := \left(i \pm \frac{1}{2}\right)h. 
\end{equation}
Note that we use the same step size for both space- and time-variables to anticipate the duality between them. Now, we integrate (\ref{eqn:MatDerInt}) over an arbitrary cell
\begin{equation}
\begin{split}
    \frac{1}{\Gamma(1-\alpha)} &\frac{d}{dt} \int_0^t (t-s)^{-\alpha} \left( \int_{x_{i-\frac{1}{2}} \mp (t-s)}^{x_{i+\frac{1}{2}} \mp (t-s)} u(y, s) dy\right) ds \\
    &\pm \frac{1}{\Gamma(1-\alpha)} \int_0^t (t-s)^{-\alpha} \left(u(x_{i+\frac{1}{2}} \mp (t-s), s) - u(x_{i-\frac{1}{2}} \mp (t-s),s)\right)ds = \int_{x_{i-\frac{1}{2}} }^{x_{i+\frac{1}{2}}} f(x,t) dx.
\end{split}
\end{equation}
Now, we would like to take a cell average of the above. To this end, divide the equation by $h$ and define
\begin{equation}
\overline{u}(x_i, t) := \frac{1}{h} \int_{x_{i-\frac{1}{2}} }^{x_{i+\frac{1}{2}}} u(x,t) dx,
\end{equation}
to obtain
\begin{equation}
\label{eqn:Conservation}
\begin{split}
    \frac{1}{\Gamma(1-\alpha)} &\frac{d}{dt} \int_0^t (t-s)^{-\alpha} \overline{u}(x_i \mp (t-s), s) ds \\
    &\pm \frac{1}{\Gamma(1-\alpha)} \int_0^t (t-s)^{-\alpha} \frac{u(x_{i+\frac{1}{2}} \mp (t-s), s) - u(x_{i-\frac{1}{2}} \mp (t-s),s)}{h}ds = \overline{f}(x_i, t).
\end{split}
\end{equation}
The finite-volume method is based on discretizing the above conservation law and numerically approximating the cell averages of the unknown $u$. Therefore, the numerical solution will be conserved and we move away from pointwise values. We only have to assume that $u$ and $f$ have well-defined spatial averages. For example, dealing with $f(x,\cdot) \propto \delta(x)$ is feasible.  

Now, in order to produce a fully discrete scheme, we have to discretize the temporal variable. This can be done using the L1 scheme which in our case is based on approximating $\overline{u}(x_i, t)$ by piecewise constant function on each time subinterval $(t_j, t_{j+1})$. That is, the integral above can be discretized as follows
\begin{equation}
\begin{split}
    \frac{1}{\Gamma(1-\alpha)} &\int_0^{t_n} (t_n-s)^{-\alpha} \overline{u}(x_i \mp (t_n-s),s) ds \\
    &= \frac{1}{\Gamma(1-\alpha)}\sum_{j=0}^{n-1} \int_{t_{j}}^{t_{j+1}} (t_n-s)^{-\alpha} \overline{u}(x_i \mp (t_n-s),s) ds = \frac{h^{1-\alpha}}{\Gamma(2-\alpha)}\sum_{j=0}^{n-1} b_{n-j} u_{i \mp (n-j)}^{j} + r_{in}(h),
\end{split}
\end{equation}
where $r_{in}(h)$ is the remainder and we denoted $u_i^k := \overline{u}(x_i, t_k)$. The coefficients are
\begin{eqnarray}
\label{eqn:Coefficients}
b_k = k^{1-\alpha} - (k-1)^{1-\alpha}.
\end{eqnarray}
The time derivative can now be approximated by the finite difference so that
\begin{equation}
\label{eqn:tDisc}
\begin{split}
    \frac{1}{\Gamma(1-\alpha)} &\frac{d}{dt} \int_0^{t_{n}} (t_{n}-s)^{-\alpha} \overline{u}(x_i \mp (t_{n}-s), s) ds \\
    &= \frac{h^{-\alpha}}{\Gamma(2-\alpha)} \left(u^{n}_{i\mp 1} + \sum_{j=0}^{n-1} b_{n-j+1} u^{j}_{i \mp (n-j+1)} - b_{n-j} u^{j}_{i \mp (n-j)}\right) + R_{in}^1(h).
\end{split}
\end{equation}
It is a classical result that for sufficiently smooth functions, say twice differentiable, this approximation is of order $2-\alpha$, that is the remainder $R_{nm}$ satisfies \cite{Li19a}
\begin{equation}
|R^1_{nm}| \leq C_u h^{2-\alpha}.
\end{equation}
Furthermore, a sharp estimate of this constant has been found in \cite{plociniczak2023linear}. We also remark that for less regular functions, the order of discretization can diminish, as will be discussed below. 

So far, the construction of the scheme did not involve the spatial variation of the unknown quantity. We now make an assumption that the pointwise value of $u$ can be approximated by its cell average
\begin{equation}
\label{eqn:upwindAssum}
\overline{u}(x_i, t) \approx
\begin{cases}
    u(x_{i + \frac{1}{2}}, t), & \text{in the "+" case}, \\
    u(x_{i - \frac{1}{2}}, t), & \text{in the "-" case}, \\
\end{cases}
\end{equation}
which is accurate with $O(h)$ as $h\rightarrow 0^+$. Observe that we approximate with different values of $u$ for different cases of fractional material derivative. This is closely associated with the direction of the characteristic. The same assumption is made when deriving the upwind method \cite{leveque2002finite}. Therefore, the discretization of the spatial part in (\ref{eqn:Conservation}) now becomes in the "+" case

\begin{equation}
\label{eqn:xDisc}
\begin{split}
    \frac{h^{-1}}{\Gamma(1-\alpha)} &\int_0^{t_n} (t_n-s)^{-\alpha} \left(\overline{u}(x_{i} - (t_n-s), s) - \overline{u}(x_{i-1} - (t_n-s),s)\right)ds \\
    &= \frac{h^{-\alpha}}{\Gamma(2-\alpha)}\sum_{j=0}^{n-1} b_{n-j} \left(u^{j}_{i-{n+j}} - u^j_{i-{n+j}-1}\right) + R_{in}^2(h),
\end{split}
\end{equation} 
and similarly in the "-" variant. Truncating the remainders and combining (\ref{eqn:tDisc}) with (\ref{eqn:xDisc}) we arrive at a discretization of the conservation form of our equation (\ref{eqn:Conservation})


\begin{equation}
\label{eqn:MatDerDisc}
{\delta_{\pm}^\alpha u^n_i := \frac{h^{-\alpha}}{\Gamma(2-\alpha)}\left[u_{i\mp 1}^n - \sum_{j=0}^{n-1}\left(b_{n-j}-b_{n-j+1}\right)u^j_{i\mp{(n-j+1)}}\right] = f^n_{i}.}
\end{equation}
Note how similar terms cancel in the "$\pm$" cases thanks to the specific choice in the upwind assumption (\ref{eqn:upwindAssum}). Although the above scheme has been devised for cell averages of $u$, we can think of $u^n_i$ as an approximation for the \emph{pointwise} values. For smooth functions, this is clearly the case. We further expect that the left-hand side of this approximates the fractional derivative to some order $r$ as $h\rightarrow 0^+$ (see \cite{Li19a,Plo21}). For the result that follows, we assume that
\begin{equation}
\label{eqn:Truncation}
\left\|\delta^\alpha_\pm u^n - \matder{\alpha}{\pm} u(\cdot,t_n) \right\|_{2,h} = \|R_n(h)\|_{2,h} \leq C h^r, \quad r > 0,
\end{equation}
where the discrete norm is defined by
\begin{equation}
\label{eqn:Norm}
\|u^n\|_{p,h} := \left(h\sum_{i\in\mathbb{Z}} |u^n_i|^p\right)^\frac{1}{p}, 
\end{equation}
where in case of continuous function we identify $u^n_i = u(x_i,t_n)$. The value of the exponent $r$ strongly depends on the regularity of the approximated solution. For example, for $C^2[0,T]$ functions we have $r=2-\alpha$, which is the upper bound of the order \cite{plociniczak2023linear}. For less regular functions, for example, the typical sub-diffusive behavior $u_m(\cdot, t) \approx t^{\alpha-m}$ with $m=1,2$ we have $r = \alpha$ (see Lemma 5.2 in \cite{stynes2017error}). The relation between $r$ and the regularity of the solution is certainly worth further study in future work. However, it is clear that, thanks to our pointwise representation of the fractional material derivative, the truncation error is inherited from the Riemann-Liouville case. 

As can also be seen from (\ref{eqn:MatDerDisc}), the nonlocality of the discretization is evident: the value at $(x_i, t_n)$ depends on all previous values located near the characteristic. The discretization introduces a nonlocal stencil that is an approximation of the line $x \mp t = C$. It is natural to ask how the above discretization performs for solving some model equation. Suppose we would like to find the solution of
\begin{equation}
\label{eqn:ModelProb}
\begin{cases}
\matder{\alpha}{\pm} u(x,t) = f(x,t), & x \in \mathbb{R}, \quad 0<t\leq T, \\
u(x,0) = 0, & x \in \mathbb{R}. 
\end{cases}
\end{equation}
Then, a natural numerical scheme for the above is $\delta^\alpha_\pm u^n_i = f^n_i$, that is by (\ref{eqn:MatDerDisc})

\begin{equation}
\label{eqn:NumMetModel}
{
u^n_{i\mp 1} = \sum_{j=0}^{n-1}\left(b_{n-j}-b_{n-j+1}\right)u^j_{i\mp(n-j+1)} + h^\alpha \Gamma(2-\alpha)f^n_i, \quad u^0_i = 0.}
\end{equation}
Below we show that the method is stable and convergent.
\begin{prop}[Stability]
\label{prop:Stability}
Let $u^n_m$ be the solution of (\ref{eqn:NumMetModel}) for $f(\cdot, t) \in L^2(\mathbb{R})$ for each $t\in[0,T]$ and $f(x, \cdot) \in L^\infty(0,T)$ for each $x\in\mathbb{R}$. Then, we have
\begin{equation}
\label{eqn:Stability}
\|u^n\|_{2,h} \leq \max_j \|f^j\|_{2,h}, \quad n \in\mathbb{N}.
\end{equation} 
\end{prop}
\begin{proof}
We will focus only on the "$+$" case. We start by multiplying (\ref{eqn:NumMetModel}) by $h u^n_{i-1}$ and summing over $i$ to obtain
\begin{equation}
\|u^n\|^2_{2,h} = \sum_{j=0}^{n-1}\left(b_{n-j}-b_{n-j+1}\right) (u^j, u^n) + h^\alpha \Gamma(2-\alpha) (f^n, u^n),
\end{equation}
where $(\cdot, \cdot)$ is the scalar product associated with the discrete second norm. Now, we use the Cauchy inequality $a b \leq (a^2+b^2)/2$ and the Schwarz inequality
\begin{equation}
\|u^n\|^2_{2,h} \leq \frac{1}{2} (1-b_{n+1}) \|u^n\|^2_{2,h} + \frac{1}{2} \sum_{j=0}^{n-1}\left(b_{n-j}-b_{n-j+1}\right) \|u^j\|^2_{2,h} + h^\alpha \Gamma(2-\alpha) \|f^n\|_{2,h}\|u^n\|_{2,h},
\end{equation}
since $b_{n-j}-b_{n-j+1} \geq 0$ and the series is telescoping. Now, we use the $\epsilon$-Cauchy inequality for the last term on the right-hand side, that is, $a b \leq (a^2/\epsilon + \epsilon b^2)/2$, here with $\epsilon = b_{n+1}h^{-\alpha}/\Gamma(2-\alpha)$ to arrive at
\begin{equation}
\|u^n\|^2_{2,h} \leq \sum_{j=0}^{n-1}\left(b_{n-j}-b_{n-j+1}\right) \|u^j\|^2_{2,h} + b_{n+1} \max_j \|f^j\|^2_{2,h}.
\end{equation}
Having the main inequality recurrence, we proceed by induction. First, for $n=1$ we have
\begin{equation}
\|u^1\|^2_{2,h} \leq (1-b_2) \|u^0\|^2_{2,h} + b_2 \max_j \|f^j\|_{2,h} = (2^{1-\alpha}-1)\max_j \|f^j\|^2_{2,h} \leq \max_j \|f^j\|^2_{2,h},
\end{equation}
by vanishing of initial condition. Next, assume that the inequality is satisfied for all $j=1,2,...,n-1$. The next step is
\begin{equation}
\|u^n\|^2_{2,h} \leq (1-b_{n+1})\max_j \|f^j\|^2_{2,h} + b_{n+1}\max_j \|f^j\|^2_{2,h} = \max_j \|f^j\|^2_{2,h} \leq (\max_j \|f^j\|_{2,h})^2,
\end{equation}
and the proof is complete. 
\end{proof}
From this, thanks to linearity, it is straightforward to prove that the method is convergent.
\begin{thm}[Convergence]
\label{thm:convergence}
Let $u(x,t)$ be the solution of (\ref{eqn:ModelProb}) and $u^n_i$ the solution of the numerical scheme (\ref{eqn:NumMetModel}). Assume that $f(\cdot, t)\in L^2(\mathbb{R})$ and the truncation error satisfies (\ref{eqn:Truncation}). Then
\begin{equation}
\|u(\cdot, t_n) - u^n\|_{2,h} \leq C h^r, 
\end{equation}
where constant $C$ depends on $\alpha$, $u$ and its derivatives, while the order $r>0$ is related to the regularity of the solution.
\end{thm}
\begin{proof}
For the "+" case, we define the error $e^n_i = u(x_i, t_n) - u^n_i$. Using (\ref{eqn:NumMetModel}) we have
\begin{equation}
\begin{split}
    e^n_{i-1} &= u(x_{i-1}, t_n) - \sum_{j=0}^{n-1} \left(b_{n-j}-b_{n-j+1}\right) u^j_{i-n+j-1} - h^\alpha\Gamma(2-\alpha) f^n_i \\
    &=u(x_{i-1}, t_n) - \sum_{j=0}^{n-1} \left(b_{n-j}-b_{n-j+1}\right) u(x_{i-1} - (t_n - t_j)) + \sum_{j=0}^{n-1} \left(b_{n-j}-b_{n-j+1}\right) e^j_{i-n+j-1} \\
    &- h^\alpha\Gamma(2-\alpha) f^n_i.
\end{split}
\end{equation}
Now, due to (\ref{eqn:Truncation}) we can write
\begin{equation}
e^n_{i-1} =h^\alpha \Gamma(2-\alpha) \left[ \matder{\alpha}{+} u(x_i,t_n) + R_{in}(h) \right] + \sum_{j=0}^{n-1}\left(b_{n-j}-b_{n-j+1}\right) e^j_{i-n+j-1} - h^\alpha\Gamma(2-\alpha) f^n_i.
\end{equation}
But $u$ is a solution of (\ref{eqn:ModelProb}) so that
\begin{equation}
e^n_{i-1} = \sum_{j=0}^{n-1} \left(b_{n-j}-b_{n-j+1}\right) e^j_{i-n+j-1} +h^\alpha \Gamma(2-\alpha) R_{in}(h),
\end{equation}
which has exactly the same form as (\ref{eqn:NumMetModel}) with $f^n_i$ replaced by $R_{in}(h)$. Using Proposition \ref{prop:Stability} yields
\begin{equation}
\|e^n\|_{2,h} \leq \|R_{n}(h)\|_{2,h} \leq C h^{r}.
\end{equation}
and completes the proof. 
\end{proof}

We would like to investigate how the upwind finite-volume scheme (\ref{eqn:MatDerDisc}) behaves when applied to the PDF problem (\ref{eqn:DensityProb}). Most importantly, we are interested in whether the conservation of probability is satisfied. Due to Theorem (\ref{thm:Density}) we know that in order for the probability to be conserved the singularity of the source is required. Therefore, we cannot expect that the numerical method will attain a higher order of convergence. Hence, we are allowed to slightly modify the scheme to provide even better numerical resolution of the conservation property possibly loosing some convergence rate that is insubstantial in our non-smooth case. This modification is done by delaying the evaluation of the source function to $t = t_{n+1}$. That is, as a numerical scheme to solve (\ref{eqn:DensityProb}) we propose the following method

\begin{equation}
\label{eqn:NumMetDensityStepAhead}
{
u^n_{i\mp 1} = \sum_{j=0}^{n-1}\left(b_{n-j}-b_{n-j+1}\right)u^j_{i\mp(n-j+1)} + h^\alpha \Gamma(2-\alpha)f^{n+1}_i, \quad u^0_i = \psi(x_i).}
\end{equation}

We claim that this scheme is conservative provided $f$ satisfies (\ref{eqn:fCond}) and the initial condition approximates Dirac delta with
\begin{equation}
\|u^0\|_{1,h} = 1.
\end{equation}
Note that we also have to bear in mind that calculating the discrete norm (\ref{eqn:Norm}) will introduce a discretization error. This brings about a small error in the conservation property. 

\begin{thm}
	\label{thm:totalProbabilityConservation}
Let $f$ be a function satisfying (\ref{eqn:fCond}) with
\begin{equation}
\label{eqn:fCondDiscrete}
\|f^n\|_{1,h} = \frac{t_{n}^{-\alpha}}{\Gamma(1-\alpha)} + \rho(h),
\end{equation}
with the quadrature error $\rho(h)$ and $t_n = nh \leq T$. Then, provided that $\|u^0\|_{1,h} = 1$ we have
\begin{equation}
\|u^n\|_{1,h} \leq 1 + T \Gamma(2-\alpha) h^{\alpha-1} \rho(h).
\end{equation}
\end{thm}
\begin{proof}
Because $f\geq 0$ we immediately infer from (\ref{eqn:NumMetDensityStepAhead}) that $u^n_j\geq 0$. Furthermore, since $b_{n-j+1} - b_{n-j}\geq 0$ we can multiply (\ref{eqn:NumMetDensityStepAhead}) by $h$ and sum over $j$ to obtain
\begin{equation}
\|u^n\|_{1,h} = \sum_{j=0}^{n-1}\left(b_{n-j}-b_{n-j+1}\right)\|u^j\|_{1,h} + h^\alpha \Gamma(2-\alpha)\|f^{n+1}\|_{1,h}
\end{equation}
Now, due to the assumption that the initial step $\|u^0\|_{1,h} =1$ and (\ref{eqn:fCondDiscrete}) we proceed to
\begin{eqnarray}
\|u^n\|_{1,h} = b_n-b_{n+1} +  \sum_{j=1}^{n-1}\left(b_{n-j}-b_{n-j+1}\right)\|u^j\|_{1,h} + \frac{1-\alpha}{(n+1)^{\alpha}} + \Gamma(2-\alpha)h^\alpha \rho(h),
\end{eqnarray}
where we have used the fact that $t_{n+1} = (n+1)h$. Now, we would like to use mathematical induction to show that
\begin{equation}
\label{eqn:Induction}
\|u^n\|_{1,h} \leq 1 + n \Gamma(2-\alpha) h^\alpha \rho(h).
\end{equation}
Because then we would have had
\begin{equation}
\|u^n\|_{1,h} \leq 1 + nh \Gamma(2-\alpha) h^{\alpha-1} \rho(h) \leq 1 + T \Gamma(2-\alpha) h^{\alpha-1} \rho(h),
\end{equation}
which is our main claim. To this end, we notice that by (\ref{eqn:Coefficients}), for $n=1$ the estimate is
\begin{equation}
\|u^1\|_{1,h} \leq 2-2^{1-\alpha} + \frac{1-\alpha}{2^{1-\alpha}} + \Gamma(2-\alpha)h^\alpha \rho(h).
\end{equation}
By elementary calculus we can easily show that the function $\alpha\mapsto 2-2^{1-\alpha} + \frac{1-\alpha}{2^{1-\alpha}}$ is positive and is bounded by $2(\ln 2 - e^{-1})/\ln 2$ from below and by $1$ from above. Hence, we have
\begin{equation}
\|u^1\|_{1,h} \leq 1 + \Gamma(2-\alpha)h^\alpha \rho(h).
\end{equation}
Now assume that the inductive assertion (\ref{eqn:Induction}) is satisfied for all $j=1,2,..., n-1$. We will prove it for $n$, that is, we have
\begin{equation}
\|u^n\|_{1,h} \leq b_n-b_{n+1} +  (1+(n-1)\Gamma(2-\alpha)h^\alpha \rho(h))\sum_{j=1}^{n-1}\left(b_{n-j+1}-b_{n-j}\right) + \frac{1-\alpha}{(n+1)^{\alpha}} + \Gamma(2-\alpha)h^\alpha \rho(h).
\end{equation}
Since the sum above is telescoping, we can further obtain
\begin{equation}
\|u^n\|_{1,h} = b_n-b_{n+1} +  (1+(n-1)\Gamma(2-\alpha)h^\alpha \rho(h))(1-b_n) + \frac{1-\alpha}{(n+1)^{\alpha}} + \Gamma(2-\alpha)h^\alpha \rho(h),
\end{equation}
or
\begin{equation}
\begin{split}
    \|u^n\|_{1,h} 
    &\leq 1-b_{n+1} + \frac{1-\alpha}{(n+1)^{\alpha}} + (n-b_n)\Gamma(2-\alpha)h^\alpha \rho(h) \\
    &\leq 1-b_{n+1} + \frac{1-\alpha}{(n+1)^{\alpha}} + n\Gamma(2-\alpha)h^\alpha \rho(h).
\end{split}
\end{equation}
Further, a simple Taylor expansion argument brings us to the estimate
\begin{equation}
b_{n+1} \geq \frac{1-\alpha}{n^\alpha} + \frac{\alpha(1-\alpha)}{2}\frac{1}{n^{\alpha+1}}.
\end{equation}
Therefore,
\begin{equation}
\|u^n\|_{1,h} \leq 1-(1-\alpha)\left(\frac{1}{n^\alpha} - \frac{1}{(n+1)^\alpha}\right) - \frac{\alpha(1-\alpha)}{2}\frac{1}{n^{\alpha+1}} + n\Gamma(2-\alpha)h^\alpha \rho(h).
\end{equation}
The term in parentheses is manifestly positive, which concludes the proof. 
\end{proof}
As we can see, the condition for the source, that is, (\ref{eqn:fCond}), grants the numerical method to conserve the probability. That is, the solution is always bounded by $1$ plus some small discretization error. For exact quadrature, there would be no $\rho$ term. 

\section{Numerical simulations and illustrations}
We would like to illustrate our results from previous sections with some numerical examples. Let us consider two examples: one benchmarking the performance of our scheme for a general model problem \eqref{eqn:ModelProb}, and the other concerning the finding probability distribution function of the scaling limit of the selected L\'evy walk scheme.

\subsection{A model problem for the material derivative}
Let us consider the solution to \eqref{eqn:ModelProb} with $u(x,0)=0$ and $f(x,t)=t^\mu$ with $\mu>0$. By \eqref{eqn:MatDerPDESol} we can easily obtain the exact solution, which is $u(x,t)=t^{\alpha+\mu}/\Gamma(1+\alpha)$. Therefore, we can compute the scheme's error in a straightforward way and determine the rate of convergence. In Figure \ref{fig:OrdersExactRegular} we present the results of the convergence estimation for the problem considered. We have calculated the convergence rates for a spatio-temporal step $h\in {\{2^{-11},2^{-10},\ldots,2^{-4}\}}$ and tested it for several values of the parameter $\alpha$, $\alpha\in\{0.1,0.25,0.5,0.75,0.9\}$. The results obtained confirm the convergence of the proposed method proved in Theorem \ref{thm:convergence} for all parameters tested. Moreover, our numerical experiments show that in case of $u(x,t)=t^{\alpha+1}/\Gamma(1+\alpha)$ the conjecture related to the convergence rate of the proposed numerical scheme under the $L^\infty$~ norm is $\min(2-\alpha,\mu+\alpha)$ with $\mu=1$. The analogous convergence rate for $u(x,t)=t^{\alpha+2}/\Gamma(1+\alpha)$ is $2-\alpha$.
\begin{figure}[!ht]
\centering
\begin{subfigure}{0.49\textwidth}
\includegraphics[scale=0.64]{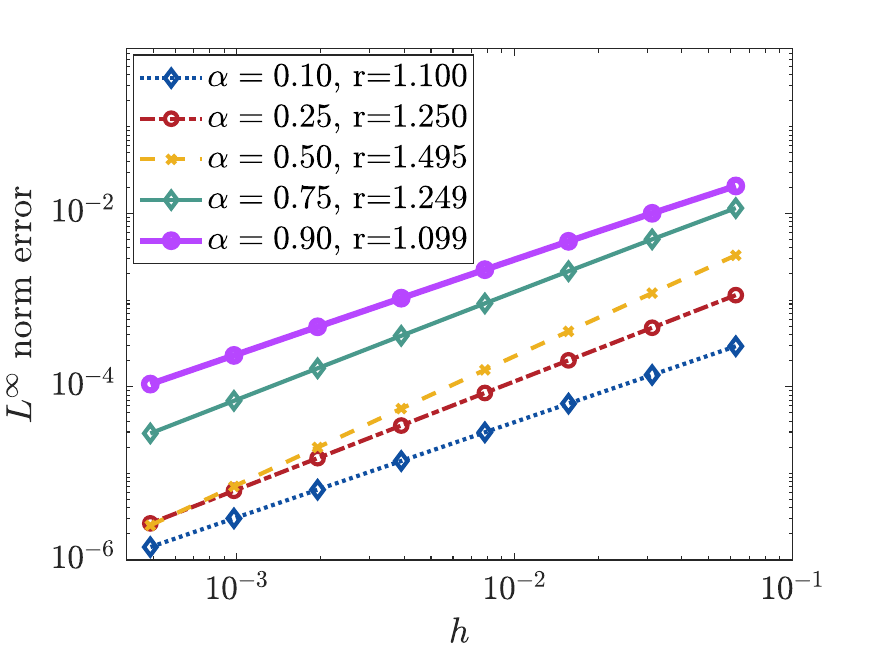}
\caption{$f(x,t)=t$, $r\approx \min(\mu+\alpha,2-\alpha)$ with $\mu=1$}
\label{fig:OrdersExactRegular_mu10}
\end{subfigure}
\begin{subfigure}{0.49\textwidth}
\includegraphics[scale=0.64]{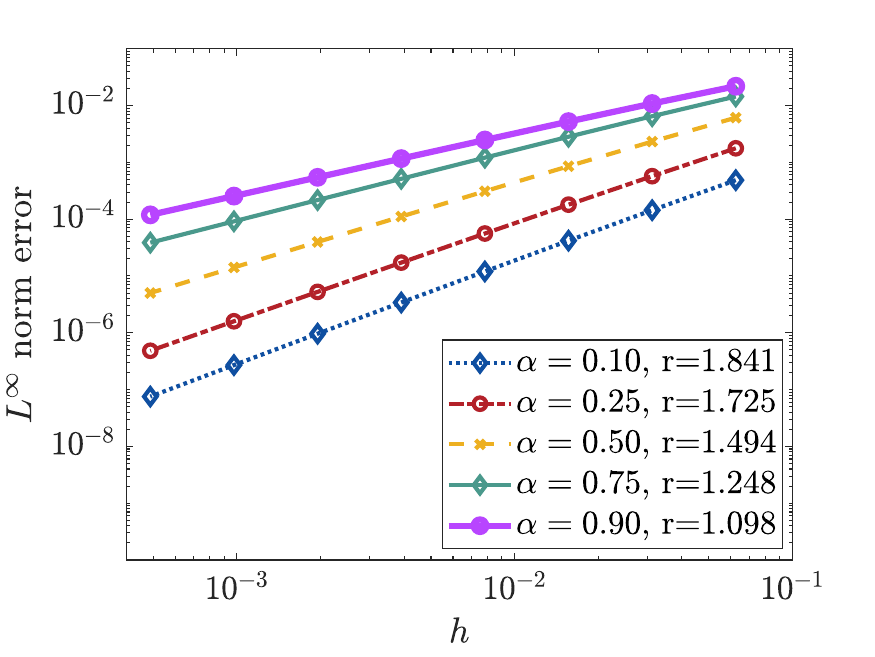}
\caption{$f(x,t)=t^2$, $r\approx 2-\alpha$}
\label{fig:OrdersExactRegular_mu12}
\end{subfigure}
\caption{The $L^\infty$ norm error for $h\in {\{2^{-11},2^{-10},\ldots,2^{-4}\}}$ and selected values of $\alpha\in\{0.1,0.25,0.5,0.75,0.9\}$ for the numerical solution to \eqref{eqn:ModelProb} with $u(x,0)=0$ and $f(x,t)=t^\mu$.}
\label{fig:OrdersExactRegular}
\end{figure}

\subsection{One-sided probability distribution function problem: a comparison with Monte Carlo approximation}
In this part, we investigate the efficiency of the numerical scheme that approximates the solution to the problem \eqref{eqn:DensityProb}, specifically a problem of PDF's derivation. Firstly, we check the performance of the proposed finite-volume upwind method in \eqref{eqn:ULWProblem} for the scaling limits of the one-sided wait-first L\'evy walk, which corresponds to the case governed by the "$+$" version of the fractional material derivative. First, in Figure \ref{fig:waitFirstLWdensity} we present the accuracy of the implemented numerical scheme using \eqref{eqn:NumMetModel}. As we can see, the numerical solution matches the exact solution. To gain more quantitative insight into the accuracy, we tested the convergence rate of the numerical scheme's error. Figure \ref{fig:OrdersWaitFirstLW} shows the convergence rates of the numerical scheme \eqref{eqn:NumMetModel} for the various values of the parameter $\alpha$, $\alpha\in\{0.10, 0.25, 0.50, 0.75, 0.90\}$ and the different spatio-temporal step sizes $h$, $h\in\{2^{-11},2^{-10},\ldots,2^{-4}\}$. For all the values of $\alpha$ examined, the estimated convergence rate order under the $L^2$ norm equals $0.5$ and therefore confirms the theoretical result in Theorem \ref{thm:convergence}. 

Next, the problem of conservation of total probability is addressed in Figure \ref{fig:probabilityConservation}. It appears that both the numerical schemes considered, the standard finite-volume upwind scheme \eqref{eqn:NumMetModel} (left panel (a)) and the modified finite-volume upwind scheme, which uses the next-step value of the source function \eqref{eqn:NumMetDensityStepAhead}, (right panel (b)), manifest the convergence of the total probability to $1$ in time. Except for the first values of the total probability in time, which are strongly affected by the singularity of the initial condition ($u(x,0)=\delta(x)$), we see that both proposed schemes differ: a standard one results in a decreasing monotonicity, while that modified one - in an increasing one. It is worth recalling that the result of the modified scheme is guaranteed by Theorem \ref{thm:totalProbabilityConservation}.

\begin{figure}[ht]
\centering
\includegraphics[scale=0.8]{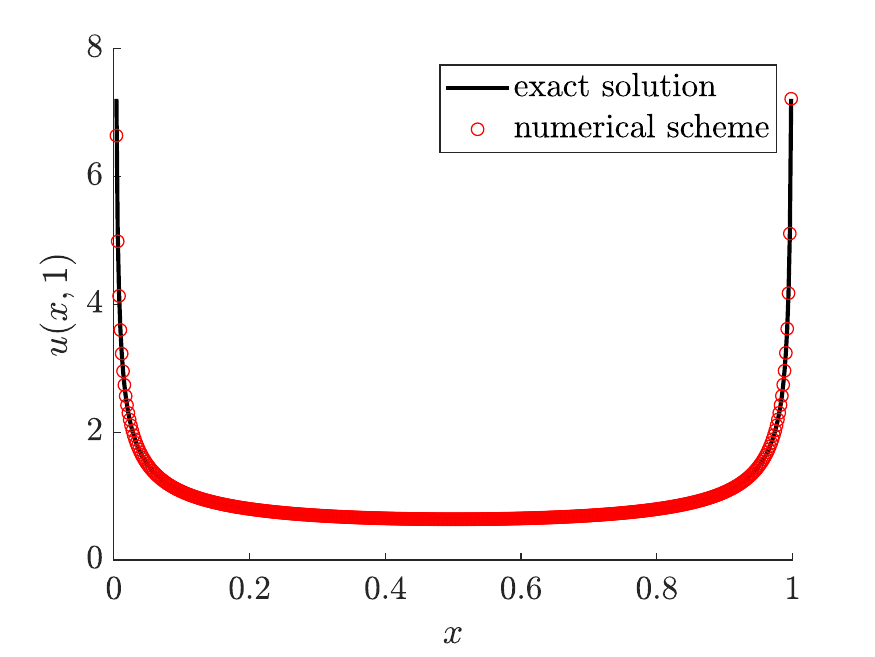}
\caption{A comparison of numerical solution obtained using  \eqref{eqn:NumMetModel} and exact solution for $\alpha=0.5$ for the PDF of the scaling limit of  the one-sided wait-first L\'evy walk, see eq. \eqref{eqn:ULWProblem} in Example \ref{ex:waitFirstLW}.}
\label{fig:waitFirstLWdensity}
\end{figure}

\begin{figure}[ht]
\centering
\includegraphics[scale=0.8]{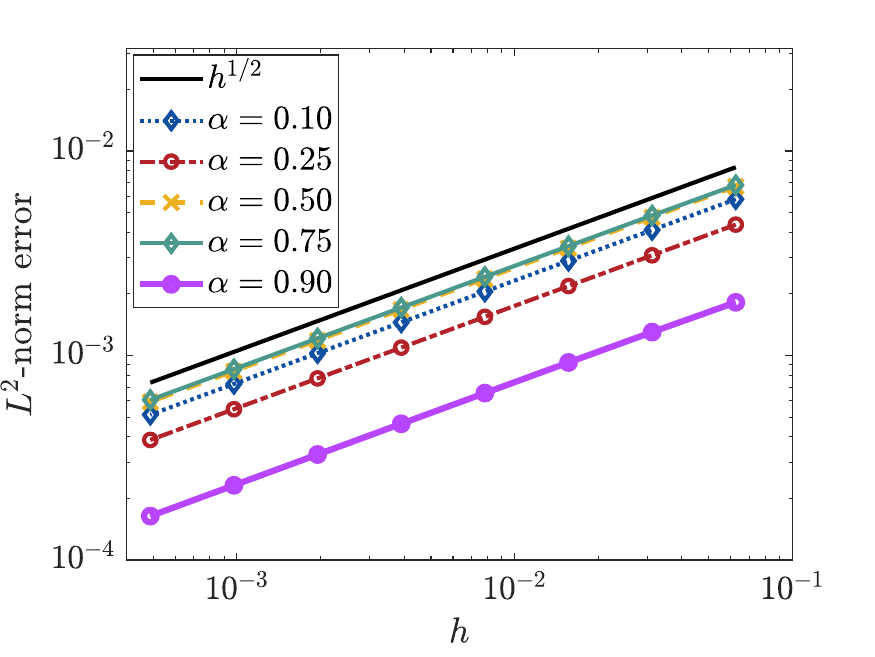}
\caption{ The $L^2$ norm error for $h\in {\{2^{-11},2^{-10},\ldots,2^{-4}\}}$ and selected values of $\alpha\in\{0.1,0.25,0.5,0.75,0.9\}$ for the numerical solution of the PDF problem of the scaling limit of the one-sided wait-first L\'evy walk, see Theorem \ref{thm:Density} and eq. \eqref{eqn:ULWProblem} in Example \ref{ex:waitFirstLW}.}
\label{fig:OrdersWaitFirstLW}
\end{figure}

\begin{figure}[ht]
\centering
\begin{subfigure}{0.49\textwidth}
\includegraphics[scale=0.64]{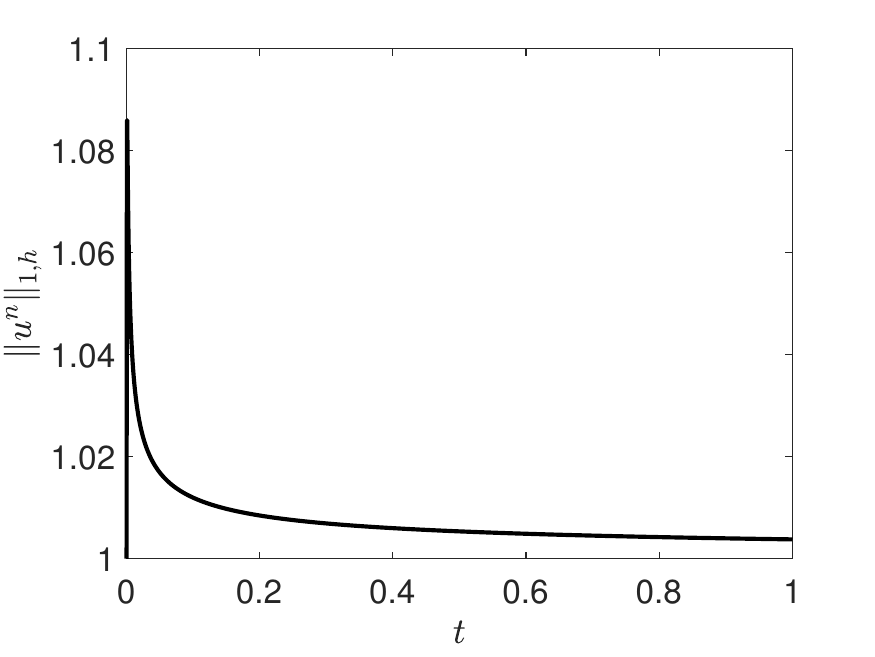}
\caption{standard finite volume upwind scheme \eqref{eqn:NumMetModel}}
\label{fig:probabilityConservationStandard}
\end{subfigure}
\begin{subfigure}{0.49\textwidth}
\includegraphics[scale=0.64]{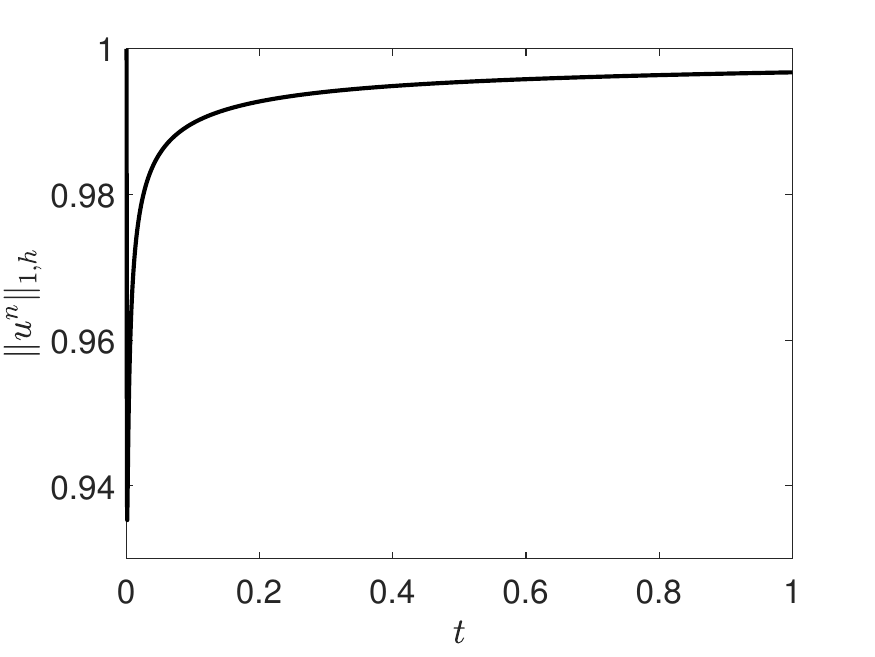}
\caption{modified finite volume upwind scheme \eqref{eqn:NumMetDensityStepAhead}}
\label{fig:probabilityConservationNextStep}
\end{subfigure}
\caption{Visualisation of the conservation law for the total probability for the consecutive time points of the numerical solution to a PDF problem of the scaling limit of the one-sided wait-first L\'evy walk (Example \ref{ex:waitFirstLW}). The parameters used in the numerical scheme are  $\alpha=0.5$, $h=2^{-11}$.}
\label{fig:probabilityConservation}
\end{figure}

Lastly, we evaluate the relative speed of our numerical method. We compare it with the Monte Carlo (MC) approximation of the respective PDF.  The main idea of the Monte Carlo approximation of the PDF of the scaling limit of the one-sided wait-first L\'evy walk given by \eqref{eqn:ULWProblem} is to simulate the trajectories of the considered stochastic process and then to estimate the PDF of the process at the fixed time point. Therefore, to simulate the trajectories of $\left(L^-_{\alpha}(S_\alpha^{-1}(t))\right)^+$ with a PDF that solves the problem in a weak sense \eqref{eqn:ULWProblem} we need to simulate the dependent trajectories of the inverse stable subordinator $S_\alpha^{-1}(t)$ and L\'evy-stable motion $L_{\alpha}(s)$, and then take their appropriate composition that leads to the trajectories of the considered scaling limits of wait-first L\'evy walk. We recall that processes ${L}_{\alpha}(s)$ and $S_\alpha(t)$ are strictly dependent, namely that the jumps occur at the same moments of time and have the same length. The algorithm can be summarized in the following points (see also \cite{Teuerle2015} for a more general problem).

\begin{itemize}
\item[1.] For a given space-time grid with a spatio-temporal step $h>0$, simulate trajectory $S_{\alpha,h}^{-1}(t)$ of the inverse $\alpha$-stable subordinator $S_\alpha^{-1}(t)$ using the following formula:
$$
S^{-1}_{\alpha,h}(t)=\left(\min\{k\in\mathbb{N}:T_\alpha(h k)>t\}-1\right)h,
$$
where the values of the subordinator $T_\alpha(t)$ are given by an exact recursive scheme
\begin{equation}
T_\alpha(h n)=T_\alpha(h (n-1))+h^{1/\alpha}\xi_n, \quad T_\alpha(0)=0.
\end{equation}
Here, the increments $\xi_n$ for $n=1,2,\ldots$ are given by an independent and identically distributed sequence of strictly positive $\alpha$-stable random variables (for details on simulations of $\alpha$-stable see \cite{Magdziarz2009}).

\item[2.] Simulate the trajectory of the parent process ${L}_{\alpha}(s)$
that is given by the recursive scheme
\begin{equation}
{L}_{\alpha,h}(h n)={L}_{\alpha,h}(h (n-1))+h^{1/\alpha}\xi_n,\quad {L}_{\alpha,h}(0)=0,
\end{equation}
where variables $\xi_n$ for $=1,2,\ldots$ are exactly the same as in point $1$ due to the strict dependence between ${L}_{\alpha}(s)$ and $S_\alpha(t)$.
\item[3.] Take the composition of ${L}_{\alpha,h}(s)$ and $S_{\alpha,h}(t)$ to obtain the trajectories of the process $\left(L^-_{\alpha}(S_\alpha^{-1}(t))\right)^+$.
\item[4.] Estimate the PDF of the process $\left(L^-_{\alpha}(S_\alpha^{-1}(t))\right)^+$ at given time point $t$ on a given space grid.
\end{itemize}
In our experiments, we have calculated the approximations of the PDF function using average values of the trajectories that ended up within the same parts of space grid as for the previous algorithm.  

In Figure \ref{fig:MCdensity} we compare the accuracy of the MC method applied to derive the PDF of the scaling limit of the one-sided wait-first L\'evy walk by plotting the estimated PDF values on the space grid for $t=1$ with the exact value obtained in \eqref{eqn:solWaitFirstLW}. As can be seen, the quality of the obtained MC approximation is rather unsatisfactory for the $10^5$ iterations and improves significantly for the $10^6$ iterations. The execution times of the MC method for $10^5$ and $10^6$ MC iterations are 23 and 283, respectively, times slower than the one using the numerical scheme proposed in this work. It is worth noticing that these results are obtained using an unparalleled implementation. The MC approximation scheme of the PDF has natural potential to be parallelized because of the independence of the trajectories of the stochastic process that can result in substantial improvement, which is not pursued in this work. Nevertheless, the optimization of both approximation approaches is worth exploring further in future work. 

\begin{figure}[!h]
\centering
\begin{subfigure}{0.49\textwidth}
\includegraphics[scale=0.64]{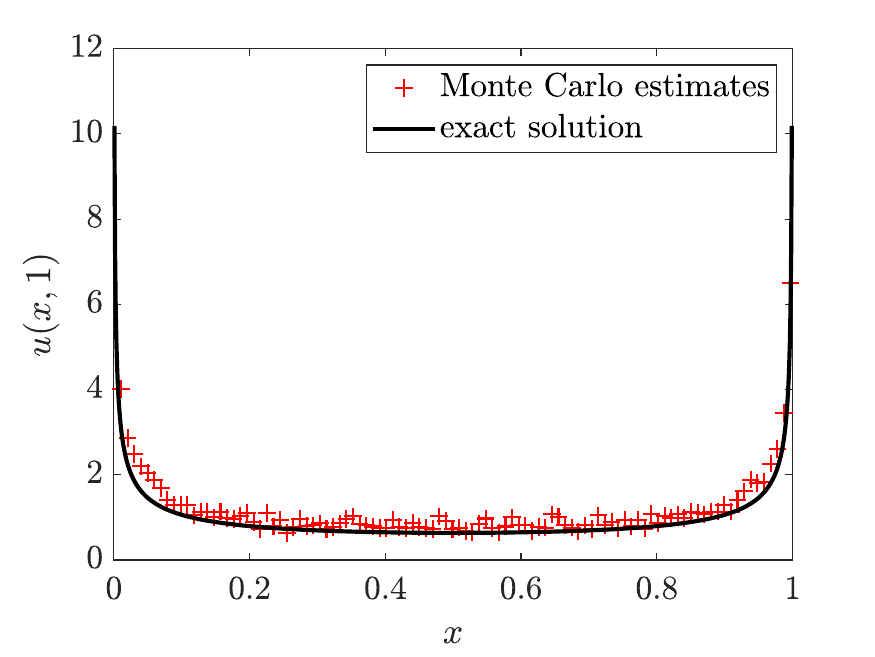}
\caption{based on $10^5$ MC trajectories}
\label{fig:MCdensity10^5}
\end{subfigure}
\begin{subfigure}{0.49\textwidth}
\includegraphics[scale=0.64]{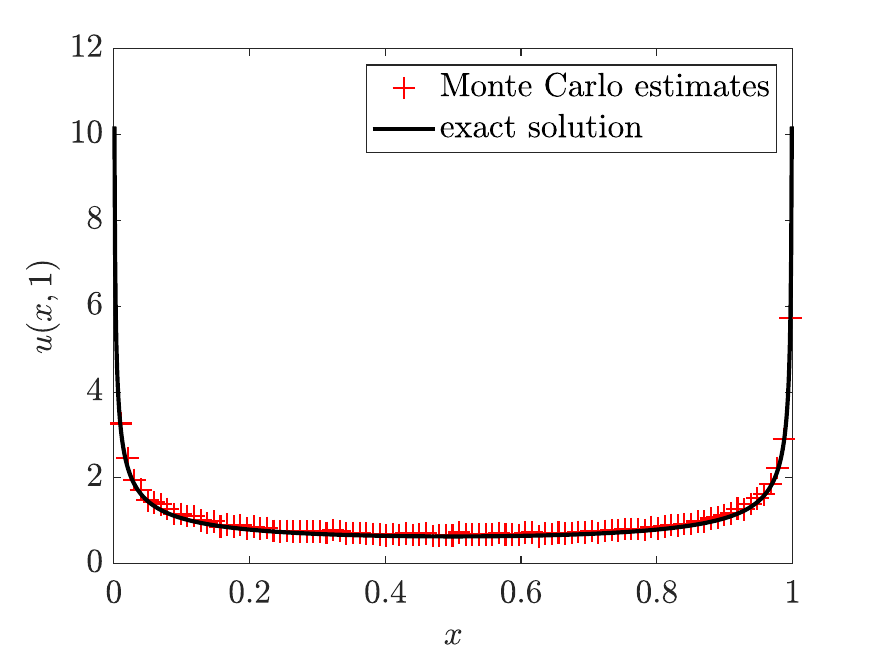}
\caption{based on $10^6$ MC trajectories}
\label{fig:MCdensity10^6}
\end{subfigure}
\caption{
A comparison of MC approximation and exact solution for $\alpha=0.5$ for the PDF of the scaling limit of the one-sided wait-first L\'evy walk, see Example \ref{ex:waitFirstLW}. The parameters used in the numerical scheme are $\alpha=0.5$, $h=2^{-10}$, while the number of MC iteration varies from $10^5$ in (a) to $10^6$ in (b).}
\label{fig:MCdensity}
\end{figure}

\section{Conclusions}

In this work, we investigate the so-called fractional material derivative, which is a generalization of the local operator of the material derivative into the fractional and non-local counterpart. First, we show in detail the origins of the fractional material derivative within the framework of L\'evy walks. We present that the considered derivative is a pseudo-differential operator that naturally occurs in the context of dynamics of the scaling limits of L\'evy walks. The main theoretical results are related to the proof of local and pointwise representation in time and space of the fractional material derivative by means of the Riemann-Liouville operator acting in the direction of $x\pm t=C$. Next, we prove the necessary and sufficent conditions for the solveability of several problems governed by the fractional material derivative, including those related to the probabilistic origins of the considered operator. Analytical solutions are obtained in three examples of one-sided probability problems of the L\'evy walks' scaling limits. The second part of this work is devoted to numerical schemes that approximate the solution of the general initial value problem that involves the fractional material derivative. We propose to use a finite-volume method that includes the L1 scheme for the temporal dimension and proved its stability together with its convergence for some model equation. In addition, we establish a modified numerical scheme that addresses the conservation of the total probability law. The convergence and accuracy of the numerical approximations have been verified for some smooth solution. Moreover, we confirm that the proposed modification of the numerical solution of the probability density function conserves the law of total probability asympotically.  

As we have presented, the stochastic approach, which is based on the stochastic CTRW dynamics, can be alternatively solved using deterministic arguments proving and recalling a deep relation between these two. Our results tentatively suggest that the proposed numerical scheme outperforms MC approximations of one-sided probability density functions of the L\'evy walks' scaling limits that are relatively slow due to simulation of subordinated stochastic process trajectories at least in its unparalleled implementation. In further work it will be relevant and interesting to consider precise estimates of the error's convergence rate with relations to a certain regularity of the solutions and devise a numerical scheme for some generalized L'evy walks \cite{LW_review}.

\section{Acknowledgement}
Ł.P. and M.A.T has been supported by the National Science Centre, Poland (NCN) under the grant Sonata Bis with a number NCN 2020/38/E/ST1/00153.

{\small 
\bibliographystyle{plain}
\bibliography{biblio}
}

\end{document}